\documentclass{article}
\usepackage[utf8]{inputenc}
\usepackage{mathpazo}
\usepackage{amsthm,amssymb}
\usepackage{mathtools}
\usepackage{todonotes}
\usepackage{array}
\usepackage[hidelinks]{hyperref}
\usepackage{enumerate}
\usepackage[left=1.1in,right=1.1in]{geometry}

\usepackage{booktabs}
\usepackage{stmaryrd}
\usepackage[capitalise]{cleveref}
\usepackage{mathtools}
\newtheorem{theorem}{Theorem}
\newtheorem{prop}[theorem]{Proposition}
\newtheorem{lemma}[theorem]{Lemma}
\newtheorem{definition}[theorem]{Definition}
\newtheorem{cor}[theorem]{Corollary}
\newtheorem{remark}[theorem]{Remark}
\newtheorem{example}[theorem]{Example}
\newcommand{\nomnom}[1]{}

\usepackage[framemethod=tikz,hidealllines=true,leftline=true]{mdframed}
\surroundwithmdframed[%
  linecolor=red!30,%
  linewidth=2pt,%
  skipabove=5pt,%
  innertopmargin=-5pt]{definition}
\surroundwithmdframed[%
  linecolor=green!30,%
  linewidth=2pt,%
  skipabove=5pt,%
  innertopmargin=-5pt]{cor}
\surroundwithmdframed[%
  linecolor=teal!40,%
  linewidth=2pt,%
  skipabove=5pt,%
  innertopmargin=-5pt]{prop}
\surroundwithmdframed[%
  linecolor=orange!40,%
  linewidth=2pt,%
  skipabove=5pt,%
  innertopmargin=-5pt]{lemma}
\surroundwithmdframed[%
  linecolor=teal!40,%
  linewidth=2pt,%
  skipabove=5pt,%
  innertopmargin=-3pt]{theorem}
\surroundwithmdframed[%
  linecolor=black!30!white,%
  linewidth=2pt,%
  skipabove=5pt,%
  innertopmargin=5pt]{proof}
\surroundwithmdframed[%
  linecolor=yellow!50,%
  linewidth=2pt,%
  skipabove=5pt,%
  innertopmargin=-5pt]{remark}
\surroundwithmdframed[%
  linecolor=purple!30,%
  linewidth=2pt,%
  skipabove=5pt,%
  innertopmargin=-5pt]{example}

\newcommand{\Z}{\mathbb{Z}}
\newcommand{\val}{\operatorname{val}}
\newcommand{\Ind}{\operatorname{Ind}}

\newcommand{\rank}{\operatorname{rank}}

\newcommand{\cous}{\operatorname{cous}}

\newcommand{\R}[1]{\mathcal{R}_{k #1}}

\title{Non-induced modular representations of cyclic groups}
\author{Liam Jolliffe\thanks{\href{mailto:ljj33@cam.ac.uk?cc=ras228@cam.ac.uk}{\texttt{ljj33@cam.ac.uk}}}\and Robert A. Spencer\thanks{\href{mailto:ras228@cam.ac.uk?cc=ljj33@cam.ac.uk}{\texttt{ras228@cam.ac.uk}}}\\\\ Department of Pure Mathematics and Mathematical Statistics, Cambridge, UK}
\date{November 2021}

\begin{document}

\maketitle
\begin{abstract}
    We compute the ring of non-induced representations for a cyclic group, $C_n$, over an arbitrary field and show that it has rank $\varphi(n)$, where $\varphi$ is Euler's totient function | independent of the characteristic of the field. Along the way, we obtain a ``pick-a-number'' trick; expressing an integer $n$ as a sum of products of $p$-adic digits of related integers.
\end{abstract}
\section{Introduction}
Given a finite group $G$ and a field $k$ of characteristic $p \ge 0$, we may study the representation theory of $G$ over $k$ via the representation ring $\mathcal{R}_{kG}$, whose elements are (isomorphism classes of) $kG$ modules, \{V\}.
Here, and throughout, the representations of $G$ over $k$ will be identified with $kG$-modules.
The addition and multiplication operations on the representation ring correspond to taking the direct sum and tensor product of modules. That is;
$$
\{V\}+\{U\}=\{V\oplus U\} \quad\text{ and }\quad \{V\}\cdot \{U\}=\{V\otimes_{k} U\}.
$$
Observe that the representation ring differs from the Grothendieck ring, as in the Grothendieck ring we have $\{V\}+\{U\}=\{W\}$ whenever there is a short exact sequence $0\to V\to W\to U\to 0$. We shall abuse notation when working in the representation ring and write $V$ for the isomorphism class $\{V\}$. 

Recall that if $H$ is a subgroup of $G$, and $V$ is a $kG$-module, then we may consider $V$ as a $kH$-module, simply by only considering the action of the elements in $kH$. This module will be denoted by $V\!\downarrow^{G}_H$, or often just  $V\!\downarrow_H$, and is called the \emph{restriction} of $V$ to $H$. It is clear that if $K\le H \le G$, then $$V\!\downarrow^{G}_H\downarrow^{H}_K=V\!\downarrow^{G}_K,$$ which is to say that restriction is transitive. 

On the other hand, if $H\le G$ and $V$ is a $kH$-module, then 
$$
V\!\uparrow_{H}^G=kG\otimes_{kH}V
$$
is the $kG$-module \emph{induced} from $H$ to $G$. Here, the $G$ action is given by left multiplication on the first factor.
We call this process \emph{induction} and say that  $V\!\uparrow_{H}^G$ is an \emph{induced module}. Induction is also transitive, which is to say if $K\le H \le G$, then $$V\!\uparrow^{H}_K\uparrow^{G}_H=V\!\uparrow^{G}_K.$$
Since $(V\oplus U)\!\uparrow_H^G\,\cong V\!\uparrow_H^G\oplus U\!\uparrow_H^G$, induction gives rise to a well defined (additive) group homomorphism of the representation rings, which we denote $\Ind_H^{G}:\R{H}\to \R{G}$.
\begin{lemma}\cite[Corollary 4.3.8 (4)]{minneapolis}
The image $\Ind_H^{G}(\R{H})$ is an ideal in $\R{G}$.
\end{lemma}
%\begin{proof}
%Let $V$ be a $kG$-module and $U$ be a $kH$-module. Then  we have the following ($kG$-module) isomorphisms: 
%\begin{align*}
%V\otimes U\uparrow_H^G &=V\otimes (kG\otimes_{kH}U)\\
%&\cong (V\otimes U) \otimes_{kH}kG\\ 
%&\cong (V\downarrow_H \otimes U)\otimes_{kH}kG\\  
%&= (V\downarrow_H \otimes U)\uparrow_H^G,
%\end{align*} 
%and thus $V\otimes U\uparrow_H^G \in \Ind_H^{G}(\R{H})$.
%
%Explicitly, there are isomorphisms $\phi: V \otimes (kG \otimes_{kH} U) \to kG \otimes_{kH} %(V\!\downarrow_H \otimes U)$
%$$\phi:v \otimes (g \otimes u) \mapsto g \otimes (g^{-1}v \otimes  u)$$
%and its inverse
%$$\phi^{-1}:gv \otimes (g \otimes u) \mapsfrom g \otimes (v \otimes u)$$
%both extended linearly.  It is a quick check to confirm these are $kG$-maps.
%%  To confirm these are $kG$ maps, we must check
%  \begin{align*}
%    h\cdot \phi\left((g^k\otimes v) \otimes w\right) &= hg^k \otimes (v \otimes g^{-k} w) \\
%    &= \phi((hg^k \otimes v) \otimes hw)\\
%    &= \phi\left(h\cdot(g^k\otimes v) \otimes w\right)
%  \end{align*}
%  and similarly
%  \begin{align*}
%    h\cdot \phi^{-1}\left(g^k\otimes (v \otimes w)\right) &= (hg^k \otimes v) \otimes h g^k w \\
%    &= \phi^{-1}\left(h g^k\otimes (v \otimes w)\right) \\
%    &=\phi^{-1}\left(h \cdot g^k\otimes (v \otimes w)\right).
%  \end{align*}
%\end{proof}
 The quotient of the representation ring by this ideal gives a measure of the $kG$-modules which are not induced from $kH$-modules. We shall now turn our attention to the $kG$-modules which are not induced from any proper subgroup. That is, we study the quotient
$$\frac{\R{G}}{\sum_{H<G}\Ind_H^G{\R{H}}},$$
which we shall refer to as the \emph{ring of non-induced representations.}
This was considered for cyclic groups by Srihari in \cite{Srihari}, where the following is proved:
\begin{theorem}\label{char0}
Let $G$ be a cyclic group of order $n$ and $k$ an algebraically closed field of characteristic 0. Then 
$$\frac{\R{G}}{\sum_{H<G}\Ind_H^G{\R{H}}}\cong \frac{\Z[Y]}{\left( \Phi_n(Y)\right)},$$
where $\Phi_n$ is the $n$th cyclotomic polynomial. In particular, 
$$
\rank\left(\frac{\R{G}}{\sum_{H<G}\Ind_H^G{\R{H}}}\right)=\varphi(n),
$$
where $\varphi(n)$ is the number of integers less than $n$ coprime to $n$.
\end{theorem}

The main goal of this paper is to prove an analogous result over fields of positive characteristic. We first shall observe that the proof of \cref{char0} provided in \cite{Srihari} actually shows:
\begin{cor}
Let $G$ be a cyclic group of order $n$ and $k$ an algebraically closed field of characteristic p such that $p\nmid n$. Then 
$$\frac{\R{G}}{\sum_{H<G}\Ind_H^G{\R{H}}}\cong \frac{\Z[X]}{\left( \Phi_n(X)\right)}.$$
In particular the rank of the ring of non-induced representations is $\varphi(n)$.
\end{cor}

\section{Cyclic \texorpdfstring{$p$}{p}-groups}
Throughout this section, let $q=p^\alpha$ and denote by $G = C_{q}$ the cyclic group of order $q$ generated by element $g$. Let $k$ be a field of characteristic $p$. 
\subsection{The representation ring}
\begin{lemma}\label{basic lemma}\cite[Proposition 1.1]{Valby}
For $k$ and $G$ as above,
\begin{enumerate}[i)]
\item there is a ring isomorphism $kG \xrightarrow[]{\sim} k\left[X\right]\big/\left(X^{q}\right)$, defined by sending $g$ to $1+X$,
\item under this isomorphism, a complete set of (pairwise non-isomorphic) $kG$-modules are given by $$\left\{ V_r := k\left[X\right]\big/\left(X^{r}\right) \;\;:\;\; 1\le r \le q \right\},$$
\item the \emph{trivial module}, $V_1$, is the unique irreducible $kG$-module.
\end{enumerate}
\end{lemma}

To understand $\R{G}$, we first need to understand the structure of multiplication.
In other words we need to understand the decomposition of the tensor product $V_r\otimes V_s$ into indecomposable parts. The decomposition rule was known to Littlewood, and has been discussed by a number of authors. The following multiplication table is given in the lecture notes of Almkvist and Fossum~\cite{Valby} and is derived by Green~\cite{manchester}:

\begin{prop}\label{rules}
For each $k<\alpha$ and $s\le p^{k+1}$ we have the following decompositions.
If $s\le p^k$, then 
\begin{align*}
V_{p^k-1}\cdot V_s &= (s-1)V_{p^k} + V_{p^k-s}\\
V_{p^k+1}\cdot V_s &= (s-1)V_{p^k} + V_{p^k+s}.
\end{align*}
Otherwise, if $p^k<s\le p^{k+1}$ then write $s=s_0p^k+s_1$ with $0\le s_1 <p^k$. Then
\begin{equation*}
V_{p^k-1}\cdot V_s=(s_1-1)V_{(s_0+1)p^k} + V_{(s_0+1)p^k-s_1} + (p^k-s_1-1)V_{s_0p^k}.
\end{equation*}
If $s_0<p-1$, then 
\begin{equation*}
\begin{aligned}
V_{p^k+1}\cdot V_s=(s_1-1)V_{(s_0+1)p^k}+ V_{(s_0+1)p^k-s_1}+(p^k-s_1-1)V_{s_0p^k}\\ + V_{s + p^k}+ V_{s-p^k}
\end{aligned}
\end{equation*}
while if $s_0=p-1$, then 
\begin{equation*}
    V_{p^k+1}\otimes V_s=(s_1+1)V_{p^{k+1}}\oplus (p^k-s_1-1)V_{(p-1)p^k}\oplus V_{(p-2)p^k+s_1}.
\end{equation*} 
\end{prop}
This suggests the construction in \cite{Valby} which defines, for $0\le k < \alpha$, elements in $\R{G}$ denoted $\chi_k = V_{p^k+1} - V_{p^k -1}$ (where $V_0$ is interpreted as the zero representation --- the zero in $\R{G}$) so that
\begin{equation*}
    \chi_k \cdot V_s = 
% This mess courtesy of https://tex.stackexchange.com/questions/559225/align-inequalities-in-cases
\left\{
\renewcommand{\arraystretch}{1.2} % like cases does
\setlength{\arraycolsep}{0pt} % we don't want padding
\begin{array}{
  l           % value
  @{\quad\quad}    % like cases does
  r           % lower bound
  >{{}}c<{{}} % relation
  c           % variable
  >{{}}c<{{}} % relation
  l           % upper bound
}
V_{s + p^k} - V_{p^k - s}                           & 1        & \le & s & \le &  p^k \\
V_{s + p^k} + V_{s - p^k}                           & p^k      & <   & s & <   & (p-1)p^k \\
V_{s-p^k} +2 V_{p^{k+1}} - V_{2p^{k+1} - (s + p^k)} & (p-1)p^k & \le & s & <   &  p^{k+1}
\end{array}
\right.
\end{equation*}
In particular, if $1 \le j < p$ then $\chi_k \cdot V_{jp^k} = V_{(j-1)p^k}+V_{(j+1)p^k}$.

To complete the multiplication rule, Renuad \cite{PM} gives a reduction theorem which allows us to express $V_r\otimes V_s$ in terms of the tensor product of smaller modules:

\begin{theorem}[Reduction Theorem]\label{reduction theorem}
For $1\le r \le s \le p^{\beta + 1}$, with $r=r_0p^\beta+r_1$, $s=s_0p^\beta+s_1$, where $1\le \beta<\alpha$, and $0\le (r_1,s_1)<p^\beta$:
\begin{align*}
V_r\otimes V_s =&\;c_1V_{p^{\beta+1}}+\mid r_1-s_1\mid \sum_{i=1}^{d_1} V_{(s_0-r_0+2i)p^\beta}+\max(0,r_1-s_1)V_{(s_0-r_0)p^\beta+b_j}\\
& \qquad+(p^\beta-s_1-r_1)\sum_{i=1}^{d_2}V_{(s_0-r_0+2i-1)p^\beta}\\
&+\sum_{j=1}^{l}a_j\left[\sum_{i=1}^{d_1}(V_{(s_0-r_0+2i)p^\beta+b_j}+V_{(s_0-r_0+2i)p^\beta-b_j})+V_{(s_0-r_0)p^\beta+b_j}\right],
\end{align*}
where
\begin{align*}
c_1&=
\begin{cases}
0 &\text{ if } r_0+s_0<p,\\
r+s-p^{\beta+1} &\text{ if } r_0+s_0\ge p,\\
\end{cases}\\
d_1&=
\begin{cases}
r_0 &\hspace{1.0em}\text{ if } r_0+s_0<p,\\
p-s_0-1  &\hspace{1.0em}\text{ if } r_0+s_0\ge p,\\
\end{cases}\\
d_2&=
\begin{cases}
r_0 &\hspace{2.65em}\text{ if } r_0+s_0<p,\\
p-s_0 &\hspace{2.65em}\text{ if } r_0+s_0\ge p,\\
\end{cases}
\end{align*}
and $V_{r_1}\otimes V_{s_1}=\sum_{j=1}^l a_jV_{b_j}$.
\end{theorem}
This allows the tensor product to be reduced to the tensor product of smaller modules, which can be calculated via repeated applications of \cref{reduction theorem}, and finally \cref{rules}, or similar multiplication rules appearing in \cite{PM}.

To aid in exposing the structure of $\R{G}$, we adjoin elements $\mu_k^{\pm 1}$ for $0\le k <\alpha$ subject to $\chi_k=\mu_k+\mu_k^{-1}$.
  If so, then for each $0 < s < p$, we have that $\mu_k^s + \mu_k^{-s} = V_{sp^k + 1} - V_{sp^k-1}$.
With this set up Alkvist and Fossum are able to completely determine the structure of $\R{G}$ by identifying $\R{G}$ with a quotient of a polynomial ring. Before we can state their result, we will first state some identities involving the $\chi_i$ and then define some families of polynomials. 
\begin{lemma}
Let $i\le j$ and $0<s<p$, then we have
$$
\chi_i^s=\sum_{0\le \nu \le s}{\binom{s}{\nu}}(V_{\nu p^i+1}-V_{\nu p^i-1})
$$
and
$$
\chi_i\chi_j=V_{p^j+p^i+1}-V_{p^j-p^i-1}-V_{p^j+p^i-1}+V_{p^j-p^i+1}.
$$
\end{lemma}
\begin{proof}
The first fact can be verified by considering the expansion of $(\mu_i+\mu_i^{-1})^s$, while the second is obtained from applications of \cref{rules}.
\end{proof}
Now, still following \cite{Valby}, we define some families of polynomials. These are easier to state if we allow ourselves to use the language of quantum numbers which are briefly introduced here.
\begin{definition}
The \emph{quantum number}, $[n]$ for $n\in \mathbb{Z}$, are polynomials in $\mathbb{Z}[X]$ that satisfy $[0]=0$, $[1]=1$, $[2] = X$, and $[n]=[2][n-1]-[n-2]$.  We shall write $[n]_x$ for the $n$th quantum number evaluated at $X=x$. 
\end{definition}
The first quantum numbers are 
\begin{center}
  \begin{tabular}{clccl}
    \toprule
    $n$      & $[n]$                &\hspace{3em}  & $n$ & $[n]$ \\
    \midrule
    0 & $0$        &  & 4 & $X^3 - 2X$ \\
    1 & $1$        &  & 5 & $X^4 - 3X^2 + 1$ \\
    2 & $X$        &  & 6 & $X^5 - 4X^3 + 3X$ \\
    3 & $X^2 - 1$  &  & 7 & $X^6 - 5X^4 + 6X^2 - 1$\\
    \bottomrule
  \end{tabular}
\end{center}
Notice that the coefficients are such that $[n]_2 = n$.
We can also write
\begin{equation*}
  [n] = \sum_{i=0}^{\lceil n/2\rceil}{(-1)}^i \binom{n-1-i}{i}X^{n-1-2i}.
\end{equation*}
In this formulation,
\begin{equation*}
  [n+1]_{q + q^{-1}} = q^{n} + q^{-n}.
\end{equation*}

Consider the polynomials in $\Z[X_0, \ldots, X_{\alpha-1}]$,
\begin{equation*}
F_j=\left(X_j-2[p]_{X_{j-1}}-2[p-1]_{X_{j-1}}\right)[p]_{X_j}.
\end{equation*}
We are now ready to state a structure theorem for $\R{G}$.
\begin{theorem}\label{Valby}\cite[Proposition 1.6]{Valby}
The map $\mathbb{Z}[X_0,\dots,X_{\alpha-1}]\to \R{G}$ defined by $X_i\mapsto \chi_i$ induces a ring isomorphism
$$
\frac{\mathbb{Z}[X_0,\dots,X_{\alpha-1}]}{(F_0,F_1,\dots,F_{\alpha-1})}\cong \R{G}.
$$
\end{theorem}
\subsection{Induced representations}

The subgroups of $G$ are all cyclic $p$ groups generated by some power of $g$. Consider $H=\langle g^{p^{\alpha-\beta}}\rangle$, the subgroup of $G$ of order $p^{\beta}$. The group algebra $kH\subseteq kG$ is identified, under the isomorphism in \cref{basic lemma}, with $k[X^{p^{\alpha-\beta}}]/(X^{p^\alpha})$. Its indecomposable representations are again each of the form $W_i=k[X^{p^{\alpha-\beta}}]/(X^{ip^{\alpha-\beta}})$ for  $1\le i < p^{\beta}$. Inducing to obtain a $kG$-module we get:
\begin{lemma} We have an isomorphism of $kG$-modules
$$W_r\!\uparrow_H^G\, \cong V_{rp^{\alpha-\beta}}.$$
\end{lemma}
\begin{proof}
Note $kG\otimes_{kH}W_r$ is cyclic as a $kG$-module, generated by $e\otimes 1$ and has dimension $r p ^{\alpha-\beta}$.
\end{proof}
Thus $\Ind_H^G\R{H}$ consists of all $kG$-modules $V_i$ such that $i$ is divisible by $p^\beta$.
Hence, since $G$ has a unique maximal subgroup, and induction is transitive,
${\sum_{H<G}\Ind_H^G{\R{H}}}$ consists of all the $kG$-modules $V_i$ such that $i$ is divisible by  $p$.  We have thus shown:
\begin{prop}\label{ideal v structure}
Let $G = C_{q}$ the cyclic group of order $q=p^\alpha$ and let $k$ be a field of characteristic $p$. Then
$${\sum_{H<G}\Ind_H^G{\R{H}}}=\left\langle \left\{V_i \;:\; p\mid i \right\}\right\rangle_{\Z}.$$
\end{prop}
To describe the ring of non-induced representations,
$$\frac{\R{G}}{\sum_{H<G}\Ind_H^G{\R{H}}},$$
via the isomorphism in \cref{Valby} we would aim to write the $kG$-modules  $V_i$ where $p\mid i$ in terms of the $\chi_i$. Instead, we shall change basis and give an alternative description of the ideal ${\sum_{H<G}\Ind_H^G{\R{H}}}$ in terms of modules which are easier to describe by polynomials in the $\chi_i$. 
\subsection{Change of basis}
Our new basis for $\R{G}$ shall use the language of quantum numbers. Observe, from \cref{rules}, that tensoring with $V_2$ satisfies a similar relation to the quantum numbers. 
\begin{lemma}
For $r < p$, $$V_r\cdot V_2=V_{r+1}+ V_{r-1}.$$
\end{lemma}
Moreover,
\begin{cor}\label{r<p}
For $r\le p$,
$$
V_r=[r]_{\chi_0}.
$$
\end{cor}

\begin{figure}
    \centering
    \includegraphics[width=\textwidth]{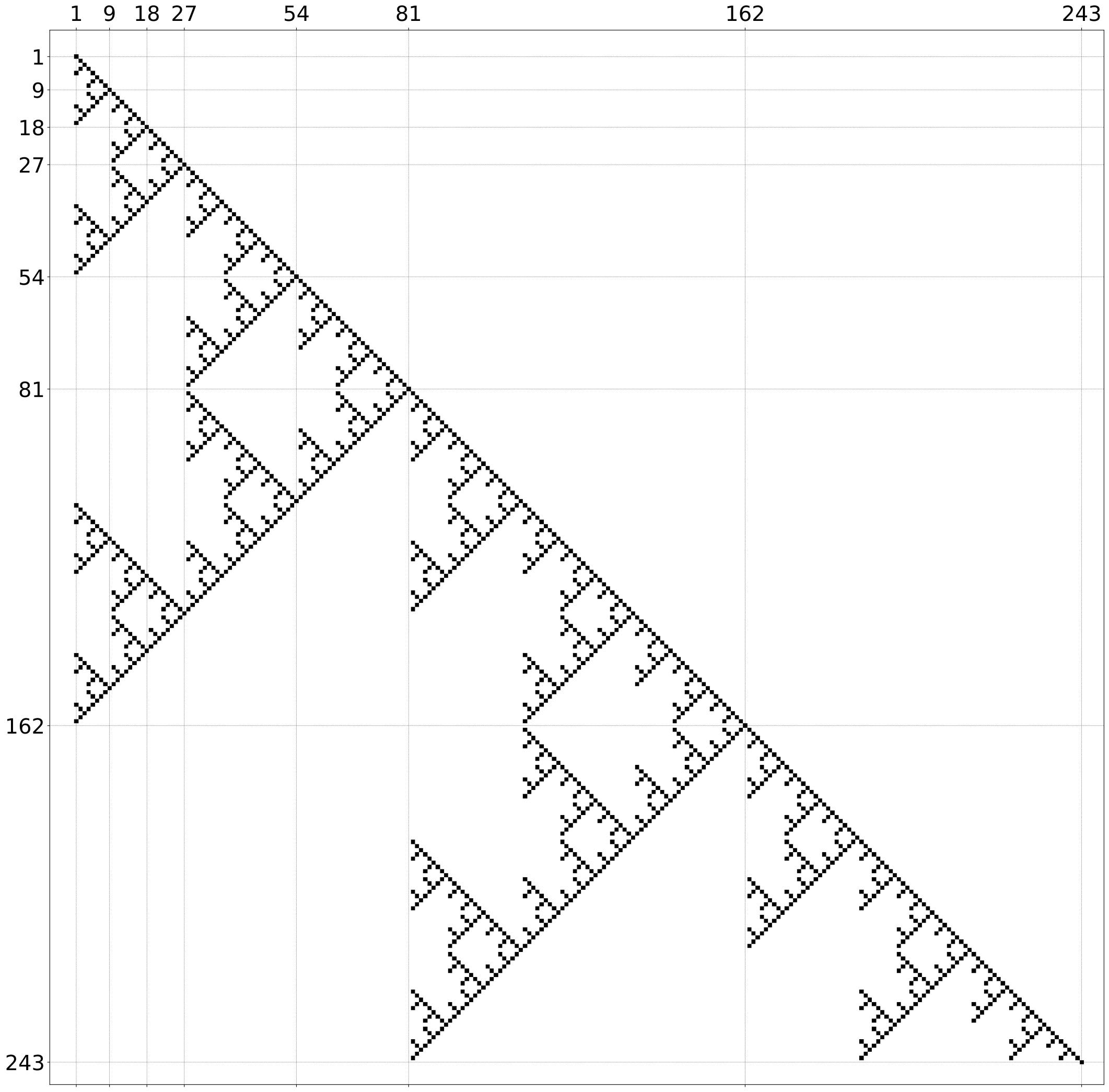}
    \caption{The change of basis matrix, expressing $V_i$ in terms of $U_j$ over characteristic $p = 3$.  Here $i$ increases downwards along rows and $j$ to the right along columns and a cell is filled if the coefficient is 1 and empty if it is 0.  Notice, for example that $V_i = U_i$ whenever $i = ap^k$ with $1 \le a < p$.  This picture appears also in \cite[Figure 2]{bonn}, where it shows the Weyl-Cartan matrix counting the decomposition multiplicities of simple $\mathfrak{sl}_2$-modules in Weyl $\mathfrak{sl}_2$-modules over characteristic $p$.}
    \label{fig:my_label}
\end{figure}

This relationship motivates defining a new basis.
\begin{definition}
Let $r< q=p^\alpha$ and write $r-1=\sum_{i=0}^{\alpha-1} r_ip^i$ with $r_i<p$. Set 
$$U_{r}:=
\prod_{i=0}^{\alpha-1}[r_i+1]_{\chi_i}.
$$ 
\end{definition}
Of course, each $U_{j}\in \R{G}$ can be written in terms of the indecomposable modules $V_i$ using repeated applications of \cref{Valby} and \cref{rules}.
The largest indecomposable appearing in this expression for $U_{r}$ comes from the term $\prod_{i=0}^{\alpha-1}{\chi_i}^{r_i} $. This largest module appearing in this term corresponds to the largest term in $\prod_{i=0}^{\alpha-1}{V_{p^i+1}}^{r_i}$.
Using the reduction theorem, \cref{Valby}, again we see that largest module appearing in our expression for $U_r$ is $V_r$.
In particular, the set $\{U_j\;:\; 0<j<q\}$ {\it is} a basis for $\R{G}$ and 
the change-of-basis matrix is lower triangular (see \cref{fig:my_label}).

We have a ring homomorphism from $\R{G}$ to $\Z$ simply by taking dimensions. As $\chi_i$ is the difference of the indecomposable modules $V_{p^i+1}$ and $V_{p^i-1}$, this homomorphism sends $\chi_i$ to 2. In particular, image of $U_r$ under the dimension homomorphism can be realised by evaluating the polynomials at $\chi_i=2$. As observed earlier, the quantum polynomials are such that $[r]_2=r$, thus the ``dimension'' of  $U_{r}$ is $\prod_{i=0}^{\alpha-1}(r_i+1).$

\begin{example}
Let $p = 5$, $\alpha = 3 $. We then have that 
\begin{align*}
    U_{12}&=[1]_{\chi_2}[3]_{\chi_1}[2]_{\chi_0}\\
    &=(1)(\chi_1^2-1)(\chi_0) \\
    &=(V_{11}-V_{9}+1)\cdot V_2\\
    &=V_{12}-V_{8}+V_{2},
\end{align*}
where the third equality follows from the identity $\chi_1^2=V_{2p+1}+V_{2p-1}+2$, and the final equality follows from \cref{rules}, which shows $V_{11}\cdot V_2=V_{12}+V_{10}$ and $V_{9}\cdot V_2=V_{10}+V_8$. Observe that if $p=5$ then $U_{12}=V_{12}-V_{8}+V_{2}$ for any $\alpha\ge 2$ as the factors $[1]_{\chi_i}$ for $i\ge 2$ make no contribution. Observe that the image of $U_{12}$ under the dimension map is $12-8+2=(1)\cdot(3)\cdot(2)$.
\end{example}
In fact, we are able to give a closed form for the $V_r$ in terms of the $U_j$.  See \cref{fig:my_label} for a visual representation of this proposition.
\begin{prop}
The exact form for the $V_i$ is as follows.
$$
V_{r} = \sum_{\substack{j\\ p^\alpha-r \in \cous{p^\alpha-j}}}U_{j}
$$
where $n = n_kp^k + n_{k-1}p^{k-1} + \cdots + n_0$ is the $p$-adic expansion of $n$, and
$$\cous(n) = \{n_k p^k \pm n_{k-1} p^{k-1} \pm \cdots \pm n_0\}. $$
\end{prop}
\begin{proof} {\it (Sketch)}
This can be shown by induction on $r$, with the base case being when $r$ is a $p$-power, and thus $V_r=U_r$.
If the equation holds for $s<r$, and $r=p^\beta+r'$ for some $\beta$ such that $r<p^{\beta+1}$, then the inductive step follows by considering 
$$
V_r=\chi_\beta \cdot V_{r'} - V_{p^\beta-r'}
$$
and
\begin{align*}
\chi_\beta \cdot V_{r'} & = \sum_{\substack{j\\ p^\alpha-r' \in \cous{p^\alpha-j}}}[2]_{\chi_\beta}U_{j}\\
& = \sum_{\substack{j\\ p^\alpha-r' \in \cous{p^\alpha-j}}}U_{p^\beta+j}-\sum_{\substack{j\\ p^\alpha-r' \in \cous{p^\alpha-j}}}U_{j-p^\beta}
\end{align*}
where $U_i=0$ for $i\le 0$. The second equality is due to the fact that $U_j$ is a product of (non-zero) quantum polynomials in the $\chi_i$. If the factor in $\chi_\beta$ is $[r]_{\chi_\beta}$, then $[2]_{\chi_\beta}[r]_{\chi_\beta}=[r+1]_{\chi_\beta}-[r-1]_{\chi_\beta}$. Moreover, our choice of $\beta$ ensures that the factor is not $[p]_{\chi_{\beta}}$.
\end{proof}
In particular, note that the $V_r$ are multiplicity-free in the $U_j$.
We can alternatively express $V_r = \mathcal{U}_r^\alpha$ where we define $\mathcal{U}_r^i$ as below. 
\begin{definition}\label{Curly U}
Let $r<q=p^\alpha$ and let $\beta\le\alpha$. Write $r=mp^\beta+j$ for $j<p^\beta$. Set $\mathcal{U}^0_r=U_r$ and define
$$
\mathcal{U}^\beta_r=
\begin{cases}
\mathcal{U}^{\beta-1}_{mp^\beta+j}+\mathcal{U}^{\beta-1}_{mp^\beta-j} & p \nmid m\\
\mathcal{U}^{\beta-1}_r & \text{else}
\end{cases}.
$$
\end{definition}
\begin{example}
For example, if $p=5$ then $V_{62} = \mathcal{U}_{62}^3 = \mathcal{U}_{62}^2$ where
\begin{align*}
\mathcal{U}^2_{62}&=\mathcal{U}^1_{62}+\mathcal{U}^1_{38}\\
&=U_{62}+U_{58}+U_{38}+U_{32}.
\end{align*}
\end{example}
\nomnom{It is clear that $\mathcal{U}^\alpha_r$ is multiplicity free, as the largest term appearing in $\mathcal{U}^{\alpha-1}_{mp^\alpha-j}$ is $U_{mp^\alpha-j}$, while all terms appearing in  $\mathcal{U}^{\alpha-1}_{mp^\alpha+j}$ have index greater than $mp^\alpha$.}
\begin{remark}
  The aforementioned dimension map enables us to play a number theoretic game.
  Indeed, recall this map sends $V_r \mapsto r$ and $U_j$ to some product of its digits plus one.
  
  Thus select prime $p$ and natural number $n$.
  Compute all of the $U_j$ appearing in $\mathcal{U}_n^\beta$ for $\beta$ such that $p^\beta > n$.
  Let all the appearing $j$ be collected in a set $J$.
  Then for each such $j\in J$, write out the $p$-adic digits of $j-1$ as $(j_0,\ldots,j_k)$.
  Finally,
  $$n = \sum_{j \in J} (j_0 + 1)(j_1 + 1) \cdots (j_k + 1).$$
  
  This is reminiscent to the ``pick-a-number'' trick played by schoolchildren. In fact, we may relax the condition that $p$ is prime in \cref{Curly U} and the trick still works, in particular we may use the usual base 10 expansion. In this situation, however, we lose the representation theoretic interpretation of this fact. 
\end{remark}
\addtocounter{theorem}{-2}
\begin{example}[continued]
\addtocounter{theorem}{1}
Continuing from our example above, if $p=5$ and $n=62$ then 
$$
J=\{62,58,38,32\}.
$$
We then write
\begin{align*}
    62-1&=2\cdot 5^2+2\cdot 5^1+1 \cdot 5^0=221_5\\
    58-1&=2\cdot 5^2+1\cdot 5^1+2 \cdot 5^0=212_5\\
    38-1&=1\cdot 5^2+2\cdot 5^1+2 \cdot 5^0=122_5\\
    32-1&=1\cdot 5^2+1\cdot 5^1+1 \cdot 5^0=111_5.
\end{align*}
Observe that
$$
62=(3)(3)(2)+(3)(2)(3)+(2)(3)(3)+(2)(2)(2),
$$
as claimed.
\end{example}
The purpose of this basis, apart from \cref{Valby}, is that it allows us to write down the ideal of induced modules very simply.
\begin{prop}
The ideal of induced modules is principle.  To be exact,
$$\sum_{H<G}\Ind_H^G{\R{H}} = \langle U_j \;:\; p \mid j\rangle_\Z = (U_p).$$
\end{prop}
\begin{proof}
Let $I = \sum_{H<G}\Ind_H^G{\R{H}}$.
Note $U_p = V_p \in I$ and each $U_{mp} = U_{(m-1)p+1} \cdot U_p$ by definition.
Hence 
$$\langle U_j \;:\; p \mid j\rangle_\Z \subseteq (U_p) = (V_p) \subseteq I = \langle V_j \;:\; p \mid j\rangle_\Z$$
where the last equality is by \cref{ideal v structure}.
But these have the same rank and hence we have equality.
\end{proof}
We are thus able to give an explicit structure to the non-induced representation ring, analogous to \cref{char0}.
\begin{cor}
The ring of non-induced representations of $C_q$ over a field of characteristic $p$ is isomorphic to 
$$
\frac{\mathbb{Z}[X_0,\dots,X_{\alpha-1}]}{([p]_{\chi_0},F_1,\dots,F_{\alpha-1})}
$$
and thus has rank $\phi(q)$.
\end{cor}
\nomnom{
\section{Aside}
In order to translate back and forth between the basis $\{V_i: 0< i <q\}$ and the basis $\{U_i: 0< i <q\}$ we need to compute the change of basis matrix. To aid this we shall define a new family of elements in $\R{G}$ which are linear combinations of the $U_i$. Moreover, these elements will be multiplicity free in the sense that the coefficient of $U_i$ will either be 1 or 0.
\begin{definition}
Let $r<q=p^\alpha$ and let $\beta<\alpha$. Write $r=mp^\beta+j$ for $j<p^\beta$. Set $\mathcal{U}^0_r=U_r$ and define
$$
\mathcal{U}^\alpha_r=
\begin{cases}
\mathcal{U}^{\alpha-1}_{mp^\alpha+j}+\mathcal{U}^{\alpha-1}_{mp^\alpha-j} & p \nmid m\\
\mathcal{U}^{\alpha-1}_r & \text{else}
\end{cases}
$$
\end{definition}

\begin{theorem}[Change of basis matrix]
Let $p^\beta\le r<p^{\beta+1} < q=p^\alpha$, then $V_r=\mathcal{U}^{\beta}_r$. 
\end{theorem}
\begin{proof}
This is clearly true when $r<p$.

\end{proof}
}
\section{General cyclic groups}
In general a cyclic group is of the form $G:=C_n= C_m\times C_{q}$, where $q=p^\alpha$ and $n=m\cdot q$ with $p\nmid m$. It is well known that the representation ring $\R{C_n}$ is the tensor product of the representation rings $\R{C_m}$ and $\R{C_{q}}$, so our task is to understand the ideal generated by induced representations.
Subgroups of $G$ are of the form $H=H_1\times H_2$ where $H_1\le C_m$ and $H_2\le C_{q}$. In particular, $H_1=C_{m'}$ for some $m'\mid m$ and $H_2=C_{p^\beta}$ for some $\beta\le \alpha$. A $kH$-module $N$ is of the form $N_1\otimes_k N_2$ where $N_i$ is a $kH_i$ module. Inducing we get $N\!\uparrow_{H}^G=N_1\!\uparrow_{H_1}^{C_m}\otimes_k N_2\!\uparrow_{H_2}^{C_q}$. In particular:
\begin{prop}
Let $G$ be a cyclic group and $k$ a field of characteristic $p$.  Suppose $G$ is of order $n=mq$, where $q=p^\alpha$ and $p\nmid m$. Let $\{V'_i\;:\; 0<i<m\}$ be the complete set of indecomposable $kC_m$-modules and let $\{V_i\;:\; 0<i<q\}$ be the complete set of indecomposable $kC_q$-modules. Then
$$
{\sum_{H<G}\Ind_H^G{\R{H}}}=\langle V'_i\otimes_k V_j \;:\; 0<i<m,\; 0<j<q,\; p\mid j \text{ or } i\mid m \rangle,
$$
which is the ideal generated by
${\sum_{H<C_m}\Ind_H^{C_m}{\R{H}}}$ and ${\sum_{H<C_q}\Ind_H^{C_q}{\R{H}}}$.
\end{prop}
It follows then that: 
\begin{theorem}
Let $G$ be a cyclic group of order $n=mq$, where $q=p^\alpha$ and $m\nmid p$. The ring of non-induced representations is isomorphic to 
$$
\frac{\mathbb{Z}[Y]}{\left( \Phi_{m}(Y)\right)} \otimes_\Z \frac{\mathbb{Z}[X_0,\dots,X_{\alpha-1}]}{([p]_{\chi_0},F_1,\dots,F_{\alpha-1})},
$$
or equivalently
$$
\frac{\mathbb{Z}[X_0,\dots,X_{\alpha-1},Y]}{\left(\Phi_m(Y),[p]_{\chi_0},F_1,\dots,F_{\alpha-1}\right)}.
$$
In particular, 
$$
\rank\left(\frac{\R{G}}{\sum_{H<G}\Ind_H^G{\R{H}}}\right)=\varphi(m)\times \varphi(p^\alpha)=\varphi(n).
$$
\end{theorem}
Note that the ranks of both the ring of non-induced representations and the ideal of induced representations are independent of the characteristic of the field, even though these sets may differ from field to field.

\section*{Acknowledgements}

The work of the first author was supported by the Woolf Fisher Trust and the Cambridge Commonwealth and International Trust. The work of the second author was supported by an EPSRC DTP Scholarship from the University of Cambridge. Both authors are grateful to Dr Stuart Martin for his guidance and support.
%\begin{thebibliography}{1}
%\bibitem{Valby} Almkvist, G., \& Fossum, R. (1978). Decomposition of exterior and symmetric powers of indecomposable $\mathbb{Z}/p\mathbb{Z}$-modules in characteristic $p$ and relations to invariants. In Séminaire d'Algèbre Paul Dubreil Proceedings, Paris 1976–1977 (30ème Année) (pp. 1-111). Springer, Berlin, Heidelberg.
%\bibitem{Srihari} Srihari, R. (2021). Non-Induced Representations of Finite Cyclic Groups. arXiv preprint arXiv:2110.07987.
%\bibitem{PM} Renaud, J. C. (1979). The decomposition of products in the modular representation ring of a cyclic group of prime power order. Journal of Algebra, 58(1), 1-11.
%\bibitem{green} J. A. Green. "The modular representation algebra of a finite group." Illinois Journal of Mathematics, 6(4) 607-619 December 1962.
% \end{thebibliography}

\bibliographystyle{halpha}
\bibliography{citations}

\begin{thebibliography}{STWZ21}

\bibitem[AF78]{Valby}
G.~Almkvist and R.~Fossum.
\newblock Decomposition of exterior and symmetric powers of indecomposable
  {${\bf Z}/p{\bf Z}$}-modules in characteristic {$p$} and relations to
  invariants.
\newblock In {\em S\'{e}minaire d'{A}lg\`ebre {P}aul {D}ubreil, 30\`eme
  ann\'{e}e ({P}aris, 1976--1977)}, volume 641 of {\em Lecture Notes in Math.},
  pages 1--111. Springer, Berlin, 1978.

\bibitem[Gre62]{manchester}
J.~A. Green.
\newblock The modular representation algebra of a finite group.
\newblock {\em Illinois J. Math.}, 6:607--619, 1962.

\bibitem[Ren79]{PM}
J.-C. Renaud.
\newblock The decomposition of products in the modular representation ring of a
  cyclic group of prime power order.
\newblock {\em J. Algebra}, 58(1):1--11, 1979.

\bibitem[Sri21]{Srihari}
R.~Srihari.
\newblock Non-induced representations of finite cyclic groups, 2021,
  \href{https://arxiv.org/abs/2110.07987}{arXiv:2110.07987}.

\bibitem[STWZ21]{bonn}
L.~Sutton, D.~Tubbenhauer, P.~Wedrich, and J.~Zhu.
\newblock Sl2 tilting modules in the mixed case, 2021,
  \href{https://arxiv.org/abs/2105.07724}{arXiv:2105.07724}.

\bibitem[Web16]{minneapolis}
P.~Webb.
\newblock {\em A course in finite group representation theory}, volume 161 of
  {\em Cambridge Studies in Advanced Mathematics}.
\newblock Cambridge University Press, Cambridge, 2016.

\end{thebibliography}
\end{document}